\documentclass{amsart}
\usepackage{enumerate}

 \newtheorem{thm}{Theorem}[section]
 \newtheorem{cor}[thm]{Corollary}
 \newtheorem{lem}[thm]{Lemma}
 
 \theoremstyle{definition}
 
 \theoremstyle{remark}

 \numberwithin{equation}{section}

\newcommand{\HH}{\mathcal{H}}
\newcommand{\VV}{\mathcal{V}}

\newcommand{\N}{\mathbb{N}}
\newcommand{\F}{\mathcal{F}}
\newcommand{\U}{\mathcal{U}}
\newcommand{\LL}{\mathcal{L}}
\newcommand{\X}{\mathfrak{X}}

\newcommand{\n}{\nabla}
\newcommand{\nn}{\widetilde{\n}}
\newcommand{\tD}{\widetilde{D}}
\newcommand{\tT}{\widetilde{T}}
\newcommand{\f}{\varphi}
\newcommand{\tg}{\widetilde{g}}
\newcommand{\tn}{\widetilde\nabla}
\newcommand{\tR}{\widetilde{R}}
\newcommand{\tF}{\widetilde{F}}
\newcommand{\tS}{\widetilde{S}}
\newcommand{\tQ}{\widetilde{Q}}

\newcommand{\al}{\alpha}

\newcommand{\ta}{\theta}
\newcommand{\om}{\omega}

\newcommand{\D}{\mathrm{d}}
\newcommand{\Span}{\mathrm{span}}
\newcommand{\Id}{\mathrm{Id}}
\newcommand{\tr}{\mathrm{tr}}
\newcommand{\Div}{\mathrm{div}}
\newcommand{\Alt}{\mathrm{Alt}}
\newcommand{\sx}{\mathop{\mathfrak{S}}\limits_{x,y,z}}

\newcommand{\thmref}[1]{Theorem~\ref{#1}}

\newcommand{\lemref}[1]{Lemma~\ref{#1}}
\newcommand{\corref}[1]{Corollary~\ref{#1}}

\begin{document}

\title[Pair of associated Schouten-van Kampen connections]
{Pair of associated Schouten-van Kampen connections adapted to an almost
contact B-metric structure}

\author[M. Manev]{Mancho Manev}

\address{%
  Department of Algebra and Geometry \\
  Faculty of Mathematics and Informatics \\
  University of Plovdiv  \\
  236, Bulgaria Blvd. \\
  Plovdiv 4000, Bulgaria
}

\email{mmanev@uni-plovdiv.bg}





\begin{abstract}
There are introduced and studied a pair of associated Schouten-van Kampen affine connections
adapted to the contact distribution and an almost contact B-metric structure generated by
the pair of associated B-metrics and their Levi-Civita connections.
By means of the constructed non-symmetric connections,
the basic classes of almost contact B-metric manifolds are characterized.
Curvature properties of the considered connections are obtained.
\end{abstract}

\maketitle

\section{Introduction}

In differential geometry of manifolds with additional tensor structures there are studied those affine connections which preserve the structure tensors and the metric, known also as natural connections on the considered manifolds.

We are interested in almost contact B-metric manifolds introduced in \cite{GaMiGr}. The geometry of three their natural connections are studied in \cite{ManGri2,Man4,Man31,ManIv38,ManIv36,Man48}.

The Schouten-van Kampen connection has been introduced for a studying of non-holonomic manifolds.
It preserves by parallelism a pair of complementary dis\-tri\-bu\-tions on a differentiable manifold
endowed with an affine connection \cite{SvK,Ia,BejFarr}. It is also used for investigations of
hyperdistributions in Riemannian manifolds (e.g., \cite{Sol}).

On the other hand, any almost contact manifold admits a hyperdistribution, the known contact distribution.
In \cite{Olsz}, it is studied the Schouten-van
Kampen connection adapted to an almost (para)contact metric structure. On these manifolds, the studied connection is not natural in general because it preserves the structure tensors except the structure endomorphism.

A counterpart of the almost contact metric structure is the almost contact B-metric structure. The B-metric (unlike the compatible metric) restricted on the contact distribution is a Norden metric, i.e. the structure endomorphism acts as an antiisometry (cf. an isometry for the compatible metric) on the contact distribution.
Other important characteristic of almost contact B-metric structure which differs it from the metric one is that the associated (0,2)-tensor of the B-metric is also a B-metric. This pair of B-metrics generates a pair of Levi-Civita connections.

In the present paper, our goal is introducing and investigation of a pair of Schou\-ten-van Kampen connections
which are associated to the pair of Levi-Civita connections and adapted to the contact distribution of an almost contact B-metric manifold. Then, we characterize the classes of considered manifolds using these connections and obtain some corresponding curvature properties.

\section{Almost Contact B-Metric Manifolds}

Let us consider an \emph{almost contact B-metric manifold} denoted by $(M,\f,\xi,\eta,g)$. This means that $M$
is a $(2n+1)$-dimensional ($n\in\N$) differentiable manifold with an almost
contact structure $(\f,\xi,\eta)$, where $\f$ is an endomorphism
of the tangent bundle $TM$, $\xi$ is a Reeb vector field and $\eta$ is its dual contact 1-form. Moreover,
$M$ is equipped with a pseu\-do-Rie\-mannian
metric $g$  of signature $(n+1,n)$, such that the following
algebraic relations are satisfied: \cite{GaMiGr}
\begin{equation*}\label{strM}
\begin{array}{c}
\f\xi = 0,\qquad \f^2 = -\Id + \eta \otimes \xi,\qquad
\eta\circ\f=0,\qquad \eta(\xi)=1,\\
g(\f x, \f y) = - g(x,y) + \eta(x)\eta(y),
\end{array}
\end{equation*}
where $\Id$ is the identity transformation. In the latter equality and further, $x$, $y$, $z$, $w$ will stand for arbitrary elements of $\X(M)$, the Lie algebra of tangent vector fields, or vectors in the tangent space $T_pM$ of $M$ at an arbitrary
point $p$ in $M$.

A classification of almost contact B-metric manifolds, which contains eleven basic classes $\F_1$, $\F_2$, $\dots$, $\F_{11}$, is given in
\cite{GaMiGr}. This classification is made with respect
to the tensor $F$ of type (0,3) defined by
\begin{equation*}\label{F=nfi}
F(x,y,z)=g\bigl( \left( \nabla_x \f \right)y,z\bigr),
\end{equation*}
where $\n$ is the Levi-Civita connection of $g$.
The following identities are valid:
\begin{equation}\label{F-prop}
\begin{array}{l}
F(x,y,z)=F(x,z,y)=F(x,\f y,\f z)+\eta(y)F(x,\xi,z)
+\eta(z)F(x,y,\xi),\\
F(x,\f y, \xi)=(\n_x\eta)y=g(\n_x\xi,y).
\end{array}
\end{equation}

The special class $\F_0$,
determined by the condition $F=0$, is the intersection of the basic classes and it is known as the
class of the \emph{cosymplectic B-metric manifolds}.

Let $\left\{e_i;\xi\right\}$ $(i=1,2,\dots,2n)$ be a basis of
$T_pM$ and $\left(g^{ij}\right)$ be the inverse matrix of the
matrix $\left(g_{ij}\right)$ of $g$. Then the following 1-forms
are associated with $F$:
\begin{equation*}\label{t}
\theta(z)=g^{ij}F(e_i,e_j,z),\quad
\theta^*(z)=g^{ij}F(e_i,\f e_j,z), \quad \omega(z)=F(\xi,\xi,z).
\end{equation*}
These 1-forms are known also as the Lee forms of the considered manifold. Obviously, the identities
$\om(\xi)=0$ and $\ta^*\circ\f=-\ta\circ\f^2$ are always valid.

Further, we use the following characteristic conditions of the
basic classes: \cite{Man8}
\begin{subequations}\label{Fi}
\begin{equation}
\begin{array}{rl}
\F_{1}: &F(x,y,z)=\frac{1}{2n}\bigl\{g(x,\f y)\ta(\f z)+g(\f x,\f
y)\ta(\f^2 z)\\
&\phantom{F(x,y,z)=\frac{1}{2n}\bigl\{}
+g(x,\f z)\ta(\f y)+g(\f x,\f z)\ta(\f^2 y)
\bigr\};\\
\F_{2}: &F(\xi,y,z)=F(x,\xi,z)=0,\quad
              \sx F(x,y,\f z)=0,\quad \ta=0;\\
\F_{3}: &F(\xi,y,z)=F(x,\xi,z)=0,\quad
              \sx F(x,y,z)=0;\\
\F_{4}: &F(x,y,z)=-\frac{1}{2n}\ta(\xi)\bigl\{g(\f x,\f y)\eta(z)+g(\f x,\f z)\eta(y)\bigr\};\\
\F_{5}: &F(x,y,z)=-\frac{1}{2n}\ta^*(\xi)\bigl\{g( x,\f y)\eta(z)+g(x,\f z)\eta(y)\bigr\};\\
\F_{6}: &F(x,y,z)=F(x,y,\xi)\eta(z)+F(x,z,\xi)\eta(y),\quad \\
                &F(x,y,\xi)=F(y,x,\xi)=-F(\f x,\f y,\xi),\quad \ta=\ta^*=0; \\
\F_{7}: &F(x,y,z)=F(x,y,\xi)\eta(z)+F(x,z,\xi)\eta(y),\quad \\
         &       F(x,y,\xi)=-F(y,x,\xi)=-F(\f x,\f y,\xi); \\
\F_{8}: &F(x,y,z)=F(x,y,\xi)\eta(z)+F(x,z,\xi)\eta(y),\quad \\
         &       F(x,y,\xi)= F(y,x,\xi)=F(\f x,\f y,\xi); \\
\F_{9}: &F(x,y,z)=F(x,y,\xi)\eta(z)+F(x,z,\xi)\eta(y),\quad \\
         &       F(x,y,\xi)=-F(y,x,\xi)=F(\f x,\f y,\xi); \\
\end{array}
\end{equation}
\begin{equation}
\begin{array}{rl}
\F_{10}: &F(x,y,z)=F(\xi,\f y,\f z)\eta(x); \\
\F_{11}:
&F(x,y,z)=\eta(x)\left\{\eta(y)\om(z)+\eta(z)\om(y)\right\}.
\end{array}
\end{equation}
\end{subequations}

Using \eqref{F-prop} and taking the traces with respect to $g$ denoted by $\tr$ and the traces with respect to $\tg$ denoted by $\tr^*$, we obtain the following relations
\begin{equation}\label{divtr}
\ta(\xi)=\Div^*(\eta),\qquad \ta^*(\xi)=\Div(\eta),
\end{equation}
where $\Div$ and $\Div^*$ denote the divergence using a trace by $g$ and by $\tg$, respectively.

As a corollary, the covariant derivative of $\xi$ with respect to $\n$ and the dual covariant derivative of $\eta$ because of $(\n_x\eta)(y)=g(\n_x\xi,y)$ are determined in each class as follows
\begin{equation}\label{Fi:nxi}
\begin{array}{l}
\F_{1}:\; \n\xi=0;\qquad\qquad
\F_{2}:\; \n\xi=0;\qquad\qquad
\F_{3}:\; \n\xi=0;\\
\F_{4}:\; \n\xi=\frac{1}{2n}\Div^*(\eta)\,\f;\qquad\qquad
\F_{5}:\; \n\xi=-\frac{1}{2n}\Div(\eta)\,\f^2;\\
%
\F_{6}:\;  g(\n_x\xi,y)=g(\n_y\xi,x)=-g(\n_{\f x}\xi,\f y),\quad \Div(\eta)=\Div^*(\eta)=0; \\
\F_{7}:\; g(\n_x\xi,y)=-g(\n_y\xi,x)=-g(\n_{\f x}\xi,\f y); \\
\F_{8}:\; g(\n_x\xi,y)=-g(\n_y\xi,x)=g(\n_{\f x}\xi,\f y); \\
\F_{9}:\; g(\n_x\xi,y)=g(\n_y\xi,x)=g(\n_{\f x}\xi,\f y); \\
\F_{10}:\; \n\xi=0; \qquad\qquad
\F_{11}:\; \n\xi=\eta\otimes\f\om^{\sharp},
\end{array}
\end{equation}
where $\sharp$ denotes the musical isomorphism of $T^*M$ in $TM$ given by $g$.

The associated metric $\tg$ of $g$ on $M$ is defined by
\(\tg(x,y)=g(x,\f y)+\eta(x)\eta(y)\).  The manifold
$(M,\f,\xi,\eta,\tg)$ is also an almost contact B-metric manifold.
The B-metric $\tg$ is also of signature $(n+1,n)$.
The Levi-Civita connection of $\tg$ is denoted by
$\nn$.
Let us denote the potential of $\nn$ regarding $\n$ by $\Phi$, i.e. $\Phi(x,y)=\nn_x y - \n_x y$. In \cite{NakGri93}, it is given a characterization of all basic classes in terms of $\Phi$ by means of
the relations between $F$ and $\Phi$ known from \cite{GaMiGr}
\begin{gather}
\begin{array}{l}\label{FPhi}
F(x,y,z)=\Phi(x,y,\f z)+\Phi(x,z,\f y)\\
\phantom{F(x,y,z)=}
+\frac12\eta(z)\{\Phi(x,y,\xi)-\Phi(x,\f y,\xi)+\Phi(\xi,x,y)-\Phi(\xi,x,\f y)\}\\
\phantom{F(x,y,z)=}
+\frac12\eta(y)\{\Phi(x,z,\xi)-\Phi(x,\f z,\xi)+\Phi(\xi,x,z)-\Phi(\xi,x,\f z)\},
\end{array}
\\
\begin{array}{l}\label{PhiF}
2\Phi(x,y,z)=-F(x,y,\f z)-F(y,x,\f z)+F(\f z,x,y)\\
\phantom{2\Phi(x,y,z)=}
+\eta(x)\{F(y,z,\xi)+F(\f z,\f y,\xi)\}\\
\phantom{2\Phi(x,y,z)=}
+\eta(y)\{F(x,z,\xi)+F(\f z,\f x,\xi)\}\\
\phantom{2\Phi(x,y,z)=}
+\eta(z)\{-F(\xi,x,y)+F(x,y,\xi)+F(x,\f y,\xi)-\omega(\f x)\eta(y)\\
\phantom{2\Phi(x,y,z)=+\eta(z)\{-F(\xi,x,y)\,}
+F(y,x,\xi)+F(y,\f x,\xi)-\omega(\f y)\eta(x)\}.
\end{array}
\end{gather}
In \cite{dissMan}, it is given the relation between $F$ and $\tF(x,y,z)=\tg(( \nn_x \f )y,z)$ as follows
\begin{equation}\label{tFF}
\begin{array}{l}
  2\tF(x,y,z)=F(\f y,z,x)-F(y,\f z,x)+F(\f z,y,x)-F(z,\f y,x)\\
  \phantom{2\tF(x,y,z)=}
  +\eta(x)\{F(y,z,\xi)+F(\f z,\f y,\xi)+F(z,y,\xi)+F(\f y,\f z,\xi)\}\\
  \phantom{2\tF(x,y,z)=}
  +\eta(y)\{F(x,z,\xi)+F(\f z,\f x,\xi)+F(x,\f z,\xi)\}\\
  \phantom{2\tF(x,y,z)=}
  +\eta(z)\{F(x,y,\xi)+F(\f y,\f x,\xi)+F(x,\f y,\xi)\}.
\end{array}
\end{equation}
Obviously, the special class $\F_0$ is determined by the following equivalent conditions: $F=0$, $\Phi=0$, $\tF=0$ and $\n=\nn$.

The properties of $\nn_x\xi$ when $(M,\f,\xi,\eta,\tg)$ is in each of the basic classes are determined in a similar way as in \eqref{Fi:nxi}.

\section{Remarkable metric connections regarding the contact distribution on the considered manifolds}

Let us consider an arbitrary almost contact B-metric manifold $(M,\f,\xi,\eta,g)$. 
%
Using the Reeb vector field $\xi$ and its dual contact 1-form $\eta$ on $M$, we determine two distributions
in the tangent bundle $TM$ of $M$ as follows
\begin{equation*}\label{HV}
  \HH=\ker(\eta),\qquad \VV=\Span(\xi).
\end{equation*}
Then the horizontal distribution $\HH$ and the vertical distribution $\VV$ form a pair of mutually complementary distributions in $TM$ which are orthogonal with respect to both of the metrics $g$ and $\tg$, i.e. $\HH\oplus\VV =TM$, $\HH\cap\VV=\{o\}$ (where $o$ is the zero vector field) and $\HH\bot\VV$. The distribution $\HH$ is known also as the contact distribution.

Let us consider the corresponding horizontal and vertical projectors $h:TM\mapsto\HH$ and $v:TM\mapsto\VV$. Bearing in mind that $x=-\f^2x+\eta(x)\xi$ for an arbitrary vector $x$ in $TM$, we use denotations $x^h$ and $x^v$ for the corresponding horizontal and vertical projections of $x$ by $h$ and $v$, respectively. Then we have
$x^h=-\f^2x$ and $x^v=\eta(x)\xi$
or equivalently
\begin{equation}\label{Xhv}
  x^h=x-\eta(x)\xi,\qquad x^v=\eta(x)\xi.
\end{equation}

\subsection{The Schouten-van Kampen connection $D$ associated to $\n$}

Let us consider the Schouten-van Kampen connection $D$ associated to $\n$
and adapt\-ed to the pair $(\HH, \VV)$. This connection is defined (locally in \cite{SvK}, see also \cite{Ia}) by
\begin{equation}\label{SvK}
  D_x y = (\n_x y^h)^h + (\n_x y^v)^v.
\end{equation}
The latter equality implies the parallelism of $\HH$ and $\VV$ with
respect to $D$.
From \eqref{Xhv} we obtain
\[
\begin{array}{l}
(\n_x y^h)^h=\n_x y -\eta(y)\n_x \xi-\eta(\n_x y)\xi, \\
(\n_x y^v)^v=\eta(\n_x y)\xi+(\n_x \eta)(y)\xi.
\end{array}
\]
Then we get the expression of the Schouten-van Kampen connection in terms of $\n$ as follows (cf. \cite{Sol})
\begin{equation}\label{SvK=n}
  D_x y = \n_x y -\eta(y)\n_x \xi+(\n_x \eta)(y)\xi.
\end{equation}

According to \eqref{SvK=n}, the potential $Q$ of $D$ with respect to $\n$ and the torsion $T$ of $D$,  defined by $Q(x,y)=D_xy-\n_xy$ and $T(x,y)= D_xy-D_yx-[x,y]$, respectively, have the following form
\begin{gather}\label{Q}
Q(x,y)=-\eta(y)\n_x \xi+(\n_x \eta)(y)\xi,
\\
\label{T}
T(x,y)=\eta(x)\n_y \xi-\eta(y)\n_x \xi+\D\eta(x,y)\xi.
\end{gather}

\begin{thm}\label{thm:D-T}
The Schouten-van Kampen connection $D$ 
 is the unique affine connection having a torsion of the form \eqref{T} and preserving the structures $\xi$, $\eta$ and the metric $g$.
\end{thm}
\begin{proof}
Taking into account \eqref{SvK=n}, we compute directly that the structures $\xi$, $\eta$ and $g$ are parallel with respect to $D$, i.e. $D\xi=D\eta=Dg=0$.
The connection $D$ preserves the metric and therefore is
completely determined by its torsion $T$.
According to  \cite{Car25}, the two spaces of all torsions and of all potentials
are isomorphic and the bijection is given as follows
\begin{gather}
T (x,y,z) = Q(x,y,z) - Q(y,x,z) ,\label{TQ}\\
2Q(x,y,z) = T (x,y,z) - T (y,z,x) + T(z,x,y).\label{QT}
\end{gather}
Then, 
the connection $D$ determined by \eqref{SvK=n} and its potential $Q$ given in \eqref{Q} are replaced in \eqref{TQ} to determine its torsion $T$ and the result is \eqref{T}. Vice versa, the form of $T$ in \eqref{T} yields by \eqref{QT} the equality for $D$ in \eqref{SvK=n}.
\end{proof}

Obviously, the connection $D$ exists on $(M,\f,\xi,\eta,g)$ in any class, but $D$ coincides with $\n$ if and only if the condition $\eta(y)\n_x \xi-(\n_x \eta)(y)\xi=0$ holds. The latter equality is equivalent to vanishing of $\n_x \xi$ for any $x$. This condition is satisfied only in the class $\F_1\oplus\F_2\oplus\F_3\oplus\F_{10}$. Let us denote this class briefly by $\U_1$. Thus, we prove the following
\begin{thm}\label{thm:D=n}
The Schouten-van Kampen connection $D$ 
coincides with $\n$ if and only if $(M,\f,\xi,\allowbreak{}\eta,g)$ belongs to the class $\U_1$.
\end{thm}

\subsection{The conditions $D$ to be natural for $(\f,\xi,\eta,g)$}

Using \eqref{SvK=n}, we express the covariant derivative of $\f$ as follows
\begin{equation}\label{Df}
  (D_x\f)y=(\n_x\f)y+\eta(y)\f\n_x\xi-\eta(\n_x\f y)\xi.
\end{equation}
Therefore, $D\f=0$ if and only if $(\n_x\f)y=-\eta(y)\f\n_x\xi+\eta(\n_x\f y)\xi$, which by \eqref{F-prop} yields
\begin{equation}\label{F_D=0}
  F(x,y,z)=F(x,y,\xi)\eta(z)+F(x,z,\xi)\eta(y).
\end{equation}
The latter condition determines the direct sum $\F_4\oplus\cdots\oplus\F_9\oplus\F_{11}$ which we denote by $\U_2$ for the sake of brevity. Thus, we find the kind of the considered manifolds where $D$ is a natural connection, i.e. the tensors of the structure $(\f,\xi,\eta,g)$ are covariantly constant regarding $D$.
In this case it follows that $(\n_x\f)\f y=\left(\n_x\eta\right)(y) \xi$ holds. Then the Schouten-van Kampen connection $D$ coincides with the $\f$B-con\-nect\-ion $\n^{*}$ defined by $\n^{*}_x y=\n_x y+\frac12\left\{(\n_x\f)\f y+(\n_x\eta)(y)\cdot \xi\right\}-\eta(y)\n_x \xi$.
Such a way, we establish the truthfulness of the following
\begin{thm}\label{thm:D-nat}
The Schouten-van Kampen connection $D$ 
is a natural connection for the structure $(\f,\xi,\eta,g)$ if and only if $(M,\f,\xi,\eta,g)$ belongs to the class $\U_2$. Then $D$ coincides with the $\f$B-connection.
\end{thm}

The $\f$B-connection is studied for all main classes $\F_1$, $\F_4$, $\F_5$, $\F_{11}$
in  \cite{ManGri1,ManGri2,Man3,Man4,ManIv37} with respect to  properties
of the torsion and the curvature as well as the conformal geometry. The restriction of the
$\f$B-connection on $\HH$ coincides with the B-connection
on the corresponding almost complex Norden manifold,
studied for their main class in  \cite{GaGrMi87}.

Let us remark that in the case when $(M,\f,\xi,\eta,g)$ belongs to a class which has a nonzero component in  both of the direct sums
$\U_1$ and $\U_2$, then
the connection $D$ is not a natural connection and it does not coincide with $\n$.
Then the class of all almost contact B-metric manifolds can be decomposed orthogonally to $\U_1\oplus\U_2$.

\subsection{The Schouten-van Kampen connection $\tD$ associated to $\nn$}


In a similar way as for $D$, let us consider the Schouten-van Kampen connection $\tD$ associated to $\nn$
and adapt\-ed to the pair $(\HH, \VV)$. This connection we define as follows
\begin{equation*}\label{tSvK}
  \tD_x y = (\nn_x y^h)^h + (\nn_x y^v)^v.
\end{equation*}
Then the hyperdistribution $(\HH, \VV)$ is parallel with
respect to $\tD$, too.
Analogously, we express the Schouten-van Kampen connection $\tD$ in terms of $\nn$ by
\begin{equation}\label{tSvK=n}
  \tD_x y = \nn_x y -\eta(y)\nn_x \xi+(\nn_x \eta)(y)\xi.
\end{equation}

By virtue of \eqref{tSvK=n}, the potential $\tQ$ of $\tD$ with respect to $\nn$ and the torsion $\tT$ of $\tD$ have the following form
\begin{gather}\label{tQ}
\tQ(x,y)=-\eta(y)\nn_x \xi+(\nn_x \eta)(y)\xi,
\\
\label{tT}
\tT(x,y)=\eta(x)\nn_y \xi-\eta(y)\nn_x \xi+\D\eta(x,y)\xi.
\end{gather}

Similarly to \thmref{thm:D-T} we have the following
\begin{thm}\label{thm:tD-tT}
The Schouten-van Kampen connection $\tD$ 
is the unique affine connection having a torsion of the form \eqref{tT} and preserving the structures $\xi$, $\eta$ and the associated metric $\tg$.
\end{thm}

It is clear that the connection $\tD$ exists on $(M,\f,\xi,\eta,\tg)$ in any class, but $\tD$ coincides with $\nn$ if and only if the condition $\eta(y)\nn_x \xi-(\nn_x \eta)(y)\xi=0$ is valid or equivalently $\nn \xi=0$. This condition holds if and only if $\widetilde{F}$ satisfies the condition \eqref{Fi} of $F$ for %
$\F_1\oplus\F_2\oplus\F_3\oplus\F_{9}$, which we denote by $\widetilde{\U}_1$. Thus, we prove the following
\begin{thm}\label{thm:tD=nn}
The Schouten-van Kampen connection $\tD$ 
coincides with $\nn$ if and only if $(M,\f,\xi,\allowbreak{}\eta,\tg)$ belongs to the class $\widetilde{\U}_1$.
\end{thm}

Taking into account \eqref{tFF}, we establish immediately the truthfulness of
\begin{lem}\label{lem:U1}
$(M,\f,\xi,\eta,g)\in\U_1$ if and only if $(M,\f,\xi,\eta,\tg)\in\widetilde{\U}_1$.
\end{lem}

Then, \thmref{thm:D=n}, \thmref{thm:tD=nn} and \lemref{lem:U1} imply the following
\begin{thm}\label{thm:D=ntD=nn}
Let $D$ and $\tD$ be the Schouten-van Kampen connections associated to $\n$ and $\nn$, respectively, and adapted to the pair $(\HH,\VV)$ on $(M,\f,\xi,\eta,g,\tg)$. Then the following assertions are equivalent:
\begin{enumerate}[(i)]
  \item $D$ coincides with $\n$;
  \item $\tD$ coincides with $\nn$;
  \item $(M,\f,\xi,\eta,g)$ belongs to $\U_1$;
  \item $(M,\f,\xi,\eta,\tg)$ belongs to $\widetilde{\U}_1$.
\end{enumerate}
\end{thm}
\begin{cor}\label{cor:D=ntD=nn}
Let $D$ and $\tD$ be the Schouten-van Kampen connections associated to $\n$ and $\nn$, respectively, and adapted to the pair $(\HH,\VV)$ on $(M,\f,\xi,\eta,g,\tg)$. If $\tD\equiv\n$ or $D\equiv\nn$ then the four connections $D$, $\tD$, $\n$ and $\nn$ coincide. The coinciding of $D$, $\tD$, $\n$ and $\nn$ is equivalent to the condition $(M,\f,\xi,\eta,g)$ and $(M,\f,\xi,\eta,\tg)$ to be cosymplectic B-metric manifolds.
\end{cor}

We obtain from \eqref{tSvK=n} the following relation between $D$ and $\tD$
\begin{equation}\label{tD=D}
  \tD_x y = D_x y + \Phi(x,y) -\eta(\Phi(x,y))\xi -\eta(y)\Phi(x,\xi).
\end{equation}

It is clear that $\tD= D$ if and only if $\Phi(x,y)=\eta(\Phi(x,y))\xi +\eta(y)\Phi(x,\xi)$ which is equivalent to $\Phi(x,y)=\eta(\Phi(x,y))\xi +\eta(x)\eta(y)\Phi(\xi,\xi)$ because $\Phi$ is symmetric.
Using relation \eqref{FPhi},
we obtain  condition \eqref{F_D=0} which determines the class $\U_2$.
Then, the following assertion is valid.
\begin{thm}\label{thm:tD=D}
The Schouten-van Kampen connections $\tD$ and $D$ associated to $\nn$ and $\n$, respectively, and adapt\-ed to the pair  $(\HH,\VV)$
coincide with each other if and only if the manifold belongs to the class $\U_2$.
\end{thm}

\subsection{The connection $\tD$ to be natural for $(\f,\xi,\eta,\tg)$}

Using \eqref{tD=D}, we have the following relation between the covariant derivatives of $\f$ with respect to $\tD$ and $D$
\begin{equation}\label{tDfiDfi}
(\tD_x\f)y=(D_x\f)y+\Phi(x,\f y)-\f\Phi(x,y)+\eta(y)\f\Phi(x,\xi)-\eta(\Phi(x,\f y))\xi.
\end{equation}
By virtue of the latter equality, we establish that $\tD\f$ and $D\f$ coincide if and only if the condition
$\Phi(x,\f^2 y,\f^2 z)=-\Phi(x,\f y,\f z)$ holds.
The latter condition is satisfied only when $(M,\f,\xi,\eta,g)$ is in the class $\F_3\oplus\U_3$, where $\U_3$ denotes the direct sum $\F_4\oplus\F_5\oplus\F_6\oplus\F_7\oplus\F_{11}$. By direct computations we establish that $(M,\f,\xi,\eta,\tg)$ belongs to the same class. Therefore, we obtain
\begin{thm}\label{thm:tDfi=Dfi}
The covariant derivatives of $\f$ with respect to the Schouten-van Kampen connections $D$ and $\tD$ 
coincide if and only if both the manifolds $(M,\f,\xi,\allowbreak{}\eta,\allowbreak{}g)$ and $(M,\f,\xi,\eta,\tg)$  belong to the class $\F_3\oplus\U_3$. 
\end{thm}

Using \eqref{PhiF}, \eqref{Df} and \eqref{tDfiDfi}, we obtain that $\tD\f=0$ is equivalent to the condition
\[
F(\f y,\f z,x)+F(\f^2 y,\f^2 z,x)-F(\f z,\f y,x)-F(\f^2 z,\f^2 y,x)=0.
\]
Then, by virtue of \eqref{Fi} we get the following
\begin{thm}\label{thm:tD-nat}
The Schouten-van Kampen connection $\tD$ 
is a natural connection for the structure $(\f,\xi,\eta,\tg)$ if and only if $(M,\f,\xi,\eta,\tg)$ belongs to the class $\F_1\oplus\F_2\oplus\U_3$. 
\end{thm}

Consequently, bearing in mind \thmref{thm:D-nat}, \thmref{thm:tDfi=Dfi}, \thmref{thm:tD-nat}, we have the validity of the following
\begin{thm}\label{thm:DtD-nat}
The Schouten-van Kampen connections $D$ and $\tD$ 
are natural connections on $(M,\f,\xi,\eta,g,\tg)$ if and only if
$(M,\f,\xi,\eta,g)$ and $(M,\f,\xi,\eta,\tg)$ belong to the class $\U_3$.
\end{thm}

\section{Torsion properties 
of the pair of connections $D$ and $\tD$}

Since $g(\xi,\xi)=1$ implies $g(\n_x\xi,\xi)=0$, then we obtain $\n_x\xi\in\HH$.
The shape operator $S:\HH\mapsto\HH$ for the metric $g$ is defined by $S(x)=-\n_x\xi$.

Then, bearing in mind the relations between $T$, $Q$ and $S$ given in \eqref{Q}, \eqref{T}, \eqref{TQ}, \eqref{QT}, the properties of the torsion, the potential and the shape operator for $D$ are related.
Analogously, similar linear relations between the torsion, the potential and the shape operator for $\tD$ are valid.

According to the expressions \eqref{Q} and \eqref{T} of $Q$ and $T$, respectively, their horizontal and vertical components have the following form
\begin{equation}\label{QThv}
\begin{array}{ll}
Q^h=-(\n\xi)\otimes\eta,\qquad & Q^v=(\n\eta)\otimes\xi,
\\
T^h=\eta\wedge(\n\xi),\qquad & T^v=\D\eta\otimes\xi.
\end{array}
\end{equation}
Then, in terms of $S$, the corresponding (0,3)-tensors $Q(x,y,z)=g(Q(x,y),z)$ and $T(x,y,z)=g(T(x,y),z)$ as well as their horizontal and vertical components are
\begin{equation*}\label{QTS03}
\begin{array}{l}
Q(x,y,z)=
-\pi_1(\xi,S(x),y,z),
\\
T(x,y,z)=
-\pi_1(\xi,S(x),y,z)+\pi_1(\xi,S(y),x,z),
\end{array}
\end{equation*}
where
\begin{equation}\label{pi1}
\pi_1(x,y,z,w)=g(y,z)g(x,w)-g(x,z)g(y,w)
\end{equation}
and
\begin{equation}\label{QTShv}
\begin{array}{ll}
Q^h=S\otimes\eta,\qquad & Q^v=-S^{\diamond}\otimes\xi,
\\
T^h=-\eta\wedge S,\qquad & T^v=-2\Alt(S^{\diamond})\otimes\xi,
\end{array}
\end{equation}
where $S^{\diamond}(x,y)=g(S(x),y)$ and $\Alt$ means the alternation.

By virtue of the equalities for the vertical components of $Q$ and $T$ in \eqref{QThv} and \eqref{QTShv}, we obtain immediately
\begin{thm}\label{thm:equiv}
The following equivalences are valid:
\begin{enumerate}
  \item[(i)]
        $\n\eta$ is symmetric                       $\Leftrightarrow$
        $\eta$ is closed, i.e. $\D\eta=0$                            $\Leftrightarrow$
        $Q^v$ is symmetric                          $\Leftrightarrow$
        $T^v$ vanishes                              $\Leftrightarrow$
        $S$ is self-adjoint regarding $g$       $\Leftrightarrow$
        $S^{\diamond}$ is symmetric                 $\Leftrightarrow$
        $M\in\U_1\oplus\F_4\oplus\F_5\oplus\F_6\oplus\F_9$;
  \item[(ii)]
        $\n\eta$ is skew-symmetric                  $\Leftrightarrow$
        $\xi$ is Killing with respect to $g$, i.e. $\LL_{\xi}g=0$       $\Leftrightarrow$
        $Q^v$ is skew-symmetric                     $\Leftrightarrow$
        $S$ is anti-self-adjoint regarding $g$           $\Leftrightarrow$
        $S^{\diamond}$ is skew-symmetric           $\Leftrightarrow$
        $M\in\U_1\oplus\F_7\oplus\F_8$;
  \item[(iii)]
        $\n\eta=0$                                  $\Leftrightarrow$
        $\D\eta=\LL_{\xi}g=0$                      $\Leftrightarrow$
        $\n\xi=0$                                   $\Leftrightarrow$
        $S=0$                                       $\Leftrightarrow$
        $S^{\diamond}=0$                                       $\Leftrightarrow$
        $D=\n$                                      $\Leftrightarrow$
        $M\in\U_1$.
\end{enumerate}
\end{thm}

The horizontal and vertical components of $\tQ$ and $\tT$ of $\tD$ are respectively
\begin{equation}\label{tQThv}
\begin{array}{ll}
\tQ^h=-(\nn\xi)\otimes\eta,\qquad & \tQ^v=(\nn\eta)\otimes\xi,
\\
\tT^h=\eta\wedge(\nn\xi),\qquad & \tT^v=\D\eta\otimes\xi.
\end{array}
\end{equation}

From $\tg(\xi,\xi)=1$ we have $\tg(\nn_x\xi,\xi)=0$ and therefore $\nn\xi\in\HH$.
The shape operator $\tS:\HH\mapsto\HH$ for the metric $\tg$ is defined by $\tS(x)=-\nn_x\xi$.

Since $(\nn_x \eta)(y)=(\n_x \eta)(y)-\eta(\Phi(x,y))$ and $\nn_x \xi=\n_x \xi +\Phi(x,\xi)$, then
\begin{equation}\label{tSSPhi}
\tS(x)=S(x)-\Phi(x,\xi),\qquad \tS^{\diamond}(x,y)=S^{\diamond}(x,\f y)-\Phi(\xi,x,\f y),
\end{equation}
where we denote $\tS^{\diamond}(x,y)=\tg(\tS(x),y)$ and $S^{\diamond}(x,y)=g(S(x),y)$.
Moreover,
\eqref{Q}, \eqref{T}, \eqref{QThv}, \eqref{tQ}, \eqref{tT} and \eqref{tQThv} imply the following relations
\begin{gather*}\label{tQQ}
\tQ(x,y)=Q(x,y)-\eta(y)\Phi(x,\xi)-\eta(\Phi(x,y))\xi,
\\
\label{tTT}
\tT(x,y)=T(x,y)+\eta(x)\Phi(y,\xi)-\eta(y)\Phi(x,\xi);
\\
\begin{array}{ll}\label{tQQTThv}
\tQ^h=Q^h-(\xi\lrcorner\Phi)\otimes\eta,\qquad & \tQ^v=Q^v-(\eta\circ\Phi)\otimes\xi,
\\
\tT^h=T^h+\eta\wedge(\xi\lrcorner\Phi),\qquad & \tT^v=T^v.
\end{array}
\end{gather*}

Using the latter equalities and \eqref{tSSPhi}, we obtain the following formulae
\begin{gather*}\label{tQQS}
\tQ=Q+(\tS-S)\otimes\eta-(\tS^{\diamond}-S^{\diamond})\otimes\xi,
\\
\label{tTTS}
\tT=T+(\tS-S)\wedge\eta;
\\
\begin{array}{ll}\label{tQQTThvS}
\tQ^h=Q^h+(\tS-S)\otimes\eta,\qquad & \tQ^v=Q^v-(\tS^{\diamond}-S^{\diamond})\otimes\xi,
\\
\tT^h=T^h+(\tS-S)\wedge\eta,\qquad & \tT^v=T^v.
\end{array}
\end{gather*}

\begin{thm}\label{thm:equiv-t}
The following equivalences are valid:
\begin{enumerate}
  \item[(i)]
        $\nn\eta$ is symmetric                       $\Leftrightarrow$
        $\eta$ is closed                              $\Leftrightarrow$
        $\tQ^v$ is symmetric                          $\Leftrightarrow$
        $\tT^v$ vanishes                              $\Leftrightarrow$
        $\tS$ is self-adjoint regarding $\tg$           $\Leftrightarrow$
        $\tS^{\diamond}$ is symmetric                 $\Leftrightarrow$
        $(M,\f,\xi,\eta,\tg)\in\widetilde\U_1\oplus\F_4\oplus\F_5\oplus\F_6\oplus\F_{10}$;
  \item[(ii)]
        $\nn\eta$ is skew-symmetric                  $\Leftrightarrow$
        $\xi$ is Killing with respect to $\tg$, i.e. $\LL_{\xi}\tg=0$       $\Leftrightarrow$
        $\tQ^v$ is skew-symmetric                     $\Leftrightarrow$
        $\tS$ is anti-self-adjoint regarding $\tg$    $\Leftrightarrow$
        $\tS^{\diamond}$ is skew-symmetric             $\Leftrightarrow$
        $(M,\f,\xi,\eta,\tg)\in\widetilde\U_1\oplus\F_7$;
  \item[(iii)]
        $\nn\eta=0$                                  $\Leftrightarrow$
        $\D\eta=\LL_{\xi}\tg=0$                     $\Leftrightarrow$
        $\nn\xi=0$                                   $\Leftrightarrow$
        $\tS=0$                                       $\Leftrightarrow$
        $\tS^{\diamond}=0$                                       $\Leftrightarrow$
        $\tD=\nn$                                      $\Leftrightarrow$
        $(M,\f,\xi,\eta,\tg)\in\widetilde\U_1$.
\end{enumerate}
\end{thm}

\section{Curvature properties of the pair of connections $D$ and $\tD$}

Let $R$ be the curvature tensor of $\n$, i.e. $R=[\n\ , \n\ ]
- \n_{[\ ,\ ]}$ and the corresponding $(0,4)$-tensor is
determined by $R(x,y,z,w)=g(R(x,y)z,w)$. The Ricci tensor
$\rho$ and the scalar curvature $\tau$  are defined as usual by
$\rho(y,z)=g^{ij}R(e_i,y,z,e_j)$ and
$\tau=g^{ij}\rho(e_i,e_j)$, where $g^{ij}$ are the corresponding components of the inverse matrix of $g$ with respect to an arbitrary basis $\{e_i\}$ ($i=1,\dots, 2n+1$)
of $T_pM$, $p\in M$.

Each non-degenerate 2-plane $\al$ in
$T_pM$ with respect to $g$ and $R$ has the following sectional curvature
$
k(\al;p)=R(x,y,y,x)(\pi_1(x,y,y,x))^{-1}, 
$ 
where $\{x,y\}$ is an arbitrary basis of $\al$.
A 2-plane $\al$ is said to be a \emph{$\xi$-section}, a \emph{$\f$-holomorphic section} or a \emph{$\f$-totally real section}
if $\xi \in \al$, $\al= \f\al$ or $\al\bot \f\al$ regarding $g$, respectively. The latter type of sections exist only for $\dim M\geq 5$.

In \cite{Man33}, some curvature properties with respect to $\n$ are studied in several subclasses of $\U_2$.

Let us denote the curvature tensor, the Ricci tensor, the scalar curvature and the sectional curvature of the connection $D$ by $R^D$, $\rho^D$, $\tau^D$ and $k^D$, respectively. The corresponding $(0,4)$-tensor is
determined by $R^D(x,y,z,w)=g(R^D(x,y)z,w)$. Analogously,
let the corresponding quantities for the connections $\tn$ and $\tD$ be denoted by $\tR$, $\widetilde\rho$, $\widetilde\tau$, $\widetilde{k}$ and $R^{\tD}$, $\rho^{\tD}$, $\tau^{\tD}$, $k^{\tD}$, respectively.
The corresponding $(0,4)$-tensors of $\tR$ and $R^{\tD}$ are obtained by $\tg$.

\begin{thm}\label{thm:KR}
The curvature tensors of 
$D$ and 
$\n$ (respectively, of $\tD$ and $\tn$) are related as follows
\begin{equation}\label{RDR}
\begin{array}{l}
  R^D(x,y,z,w)=R\left(x,y,\f^2z,\f^2w\right)+\pi_1\bigl(S(x),S(y),z,w\bigr),\\
  R^{\tD}(x,y,z,w)=\tR\left(x,y,\f^2z,\f^2w\right)+\widetilde\pi_1\bigl(\tS(x),\tS(y),z,w\bigr),
\end{array}
\end{equation}
where $\widetilde\pi_1$ is constructed as in \eqref{pi1} by $\tg$.
\end{thm}
\begin{proof}
Using \eqref{SvK=n}, we compute $R^D$. Taking into account that $g(\n_x\xi,\xi)=0$ for any $x$ and $D\xi=0$, we obtain the equality
\[
\begin{array}{l}
  R^D(x,y)z=R(x,y)z-\eta(z)R(x,y)\xi-\eta(R(x,y)z)\xi\\
  \phantom{R^D(x,y)z=}
  -g\left(\n_x\xi,z\right)\n_y\xi+g\left(\n_y\xi,z\right)\n_x\xi.
\end{array}
\]
The latter equality implies the first relation in \eqref{RDR}.

The second equality in \eqref{RDR} follows as above but in terms of $\tD$, $\tn$ and their corresponding metric $\tg$.
\end{proof}

Then, \thmref{thm:KR} has the following
\begin{cor}\label{cor:Ric}
The Ricci tensors of 
$D$ and
$\n$ (respectively, of $\tD$ and $\tn$) are related as follows
\begin{equation}\label{RicDRic}
\begin{array}{l}
  \rho^D(y,z)=\rho(y,z)-\eta(z)\rho(y,\xi)-R(\xi,y,z,\xi)\\
  \phantom{\rho^D(y,z)=\rho(y,z)}
-g(S(S(y)),z)+\tr(S)g(S(y),z),\\
  \rho^{\tD}(y,z)=\widetilde\rho(y,z)-\eta(z)\widetilde\rho(y,\xi)-\tR(\xi,y,z,\xi)\\
\phantom{  \rho^{\tD}(y,z)=\widetilde\rho(y,z)}
-\tg(\tS(\tS(y)),z)+\widetilde\tr(\tS)\tg(\tS(y),z),
\end{array}
\end{equation}
where $\widetilde\tr$ denotes the trace with respect to $\tg$.
\end{cor}

Let us remark that we have
$
\tr(S)=\widetilde\tr(\tS)=-\Div(\eta),
$
because of \eqref{divtr}, the definitions of $S$ and $\tS$ as well as $g^{ij}\Phi(\xi,e_i,e_j)=0$, using \eqref{PhiF} and \eqref{F-prop}.

From the definition of the shape operator we get
$
R(x,y)\xi =-\left(\n_x S\right)y+\left(\n_y S\right)x.
$
Then, the latter formula and $S(\xi)=-\n_{\xi} \xi=-\f\omega^{\sharp}$ lead to the following expression of one of the components in the right-hand side of \eqref{RicDRic}
\[
R(\xi,y,z,\xi) = g\bigl(\left(\n_{\xi} S\right)y-\left(\n_y S\right)\xi,z\bigr)
=g\bigl(\left(\n_{\xi} S\right)y-\n_y S(\xi)- S(S(y)),z\bigr).
\]
Therefore, taking the trace of the latter equalities and using for the divergence of the 1-form $\omega\circ\f$ the relations
$\Div(\omega\circ\f)=g^{ij}\left(\n_{e_i} \omega\circ\f\right)e_j=g^{ij}g\left(\n_{e_i}\f \omega^{\sharp},e_j\right)=-\Div(S(\xi))$,
we obtain
\begin{equation}\label{ricxixi}
\rho(\xi,\xi)= \tr(\n_{\xi}S)-\Div(S(\xi))-\tr(S^2).
\end{equation}
Similar equalities for the quantities with a tilde are valid with respect to $\tg$, i.e.
\begin{equation}\label{tricxixi}
\widetilde\rho(\xi,\xi)= \widetilde\tr(\nn_{\xi}\tS)-\widetilde\Div(\tS(\xi))-\widetilde\tr(\tS^2).
\end{equation}
Bearing in mind the latter computations, from \corref{cor:Ric} we obtain the following
\begin{cor}\label{cor:tau}
The scalar curvatures of 
$D$ and 
$\n$ (respectively, of $\tD$ and $\tn$) are related as follows
\begin{equation*}\label{tauDtau}
\begin{array}{l}
  \tau^D=\tau-2\rho(\xi,\xi)-\tr(S^2) + (\tr(S))^2,\\
  \tau^{\tD}=\widetilde\tau-2\widetilde\rho(\xi,\xi)-\widetilde\tr(\tS^2) + (\widetilde\tr(\tS))^2,
\end{array}
\end{equation*}
where $\rho(\xi,\xi)$ and $\widetilde\rho(\xi,\xi)$ are expressed by $S$ and $\tS$ in \eqref{ricxixi} and \eqref{tricxixi}, respectively.
\end{cor}

From \thmref{thm:KR} we obtain the following
\begin{cor}\label{cor:kDk}
The sectional curvatures of an arbitrary 2-plane $\alpha$ at $p\in M$ regarding 
$D$ and 
$\n$ (respectively, of $\tD$ and $\tn$) are related as follows
\begin{equation}\label{kDk}
\begin{split}
  k^D(\al;p)&=k(\al;p)\\
  &\phantom{=}+\frac{\pi_1(S(x),S(y),y,x)-\eta(x)R(x,y,y,\xi)-\eta(y)R(x,y,\xi,x)}{\pi_1(x,y,y,x)},\\
  k^{\tD}(\al;p)&=\widetilde k(\al;p)\\
  &\phantom{=}+\frac{\widetilde\pi_1(\tS(x),\tS(y),y,x)-\eta(x)\tR(x,y,y,\xi)
  -\eta(y)\tR(x,y,\xi,x)}{\widetilde\pi_1(x,y,y,x)},
\end{split}
\end{equation}
where $\{x,y\}$ is an arbitrary basis of $\alpha$.
\end{cor}

If $\al$ is a $\xi$-section  at $p\in M$ denoted by $\al_{\xi}$ and $\{x,\xi\}$ is its basis, then from \eqref{kDk} and $g(S(x),\xi)=0$ for any $x$ we obtain that the sectional curvature of $\alpha$ regarding
$D$ is zero, i.e.
$
k^D(\al_{\xi};p)=0.
$
Analogously, we have $k^{\tD}(\al_{\xi};p)=0$.

If $\al$ is a $\f$-section  at $p\in M$ denoted by $\al_{\f}$ and $\{x,y\}$ is its arbitrary basis, then from \eqref{kDk} and $\eta(x)=\eta(y)=0$ we obtain that the sectional curvatures of $\al_{\f}$ regarding $D$ and $\n$ are related as follows
\[
k^D(\al_{\f};p)=k(\al_{\f};p)
+\frac{\pi_1(S(x),S(y),y,x)}{\pi_1(x,y,y,x)}.
\]
Analogously, we have
\[
k^{\tD}(\al_{\f};p)=\widetilde k(\al_{\f};p)
+\frac{\widetilde\pi_1(\tS(x),\tS(y),y,x)}{\widetilde\pi_1(x,y,y,x)}.
\]

If $\al$ is a $\f$-totally real section orthogonal to $\xi$ denoted by $\al_{\bot}$ and $\{x,y\}$ is its arbitrary basis, then from \eqref{kDk} and $\eta(x)=\eta(y)=0$ we obtain that the sectional curvatures of $\al_{\bot}$ regarding $D$ and $\n$ are related as follows
\[
k^D(\al_{\bot};p)=k(\al_{\bot};p)
+\frac{\pi_1(S(x),S(y),y,x)}{\pi_1(x,y,y,x)}.
\]
Analogously, we have
\[
k^{\tD}(\al_{\bot};p)=\widetilde k(\al_{\bot};p)
+\frac{\widetilde \pi_1(\tS(x),\tS(y),y,x)}{\widetilde \pi_1(x,y,y,x)}.
\]
In the case when $\al$ is a $\f$-totally real section non-orthogonal to $\xi$ regarding $g$ or $\tg$, the relation between the the corresponding sectional curvatures regarding $D$ and $\n$ (respectively, $\tD$ and $\nn$) is just the first (respectively, the second) equality in \eqref{kDk}.

The equalities in the present section are specialised for the considered manifolds in the different classes since $S$ and $\tS$ have a special form in each class, bearing in mind \eqref{Fi:nxi}.

\end{document}